\documentclass[a4paper,12pt]{article}
\usepackage{comment}
\usepackage{cite}
\usepackage{amsmath,amssymb,amsfonts}
\usepackage[T1]{fontenc}
\usepackage[utf8]{inputenc}
\usepackage{graphicx}
\usepackage{fancyhdr}
\usepackage{float}
\usepackage{xcolor}
\usepackage{authblk}
\usepackage{mathrsfs}
\usepackage{empheq}
\usepackage[hyphens]{url}
\usepackage{hyperref}
\usepackage{amsthm}
\usepackage{tikz}
\usetikzlibrary{decorations.markings}
\usepackage{enumitem}
\usepackage{geometry}

\theoremstyle{plain}
\newtheorem{theorem}{Theorem}[section]
\newtheorem{proposition}[theorem]{Proposition}
\newtheorem{lemma}[theorem]{Lemma}
\newtheorem{corollary}[theorem]{Corollary}

\theoremstyle{definition}
\newtheorem{definition}[theorem]{Definition}

\theoremstyle{remark}
\newtheorem{remark}[theorem]{Remark}

\theoremstyle{conjecture}

\newtheorem{example}{Example}[section]

\begin{document}

\title{Toeplitz matrices from permutation displacements and the triangular kernel}

\author[$\dagger$]{Jean-Christophe {\sc Pain}$^{1,2,}$\footnote{jean-christophe.pain@cea.fr}\\
\small
$^1$CEA, DAM, DIF, F-91297 Arpajon, France\\
$^2$Universit\'e Paris-Saclay, CEA, Laboratoire Mati\`ere en Conditions Extr\^emes,\\ 
F-91680 Bruy\`eres-le-Ch\^atel, France
}

\maketitle

\begin{abstract}
Toeplitz matrices arise naturally in harmonic analysis, operator theory, and numerical analysis. In this note we investigate Toeplitz matrices whose coefficients depend on the matrix size through a scaled kernel $a_k=f(k/n)$. We show that the empirical mean of their eigenvalues converges to a weighted integral of $f$, where the weight $1-|x|$ reflects the density of diagonals in Toeplitz matrices. We then introduce a combinatorial construction associating a Toeplitz matrix to a permutation via its displacement counts. For a uniformly random permutation, the expected matrix converges to the Toeplitz matrix generated by the triangular kernel $1-|x|$. Interestingly, the triangular kernel also appears as the
covariance function of the integrated Brownian motion,
providing a probabilistic interpretation of the same operator. 
Finally, we analyze the integral operator with kernel $(1-|x-y|)$ on $[0,1]$ and determine its eigenfunctions and eigenvalues explicitly. This operator describes the limiting spectral structure associated with the averaged Toeplitz matrices arising from permutation displacements. These results highlight a natural bridge between Toeplitz matrix theory, permutation statistics, and classical integral operators.
\end{abstract}

\section{Introduction}

Toeplitz matrices are classical and fundamental objects in mathematics, named after Otto Toeplitz, who studied them extensively in the early 20th century~\cite{Szego75,Grenander58}. They are defined by the property that their entries are constant along diagonals:
\[
T_n = (a_{i-j})_{1 \le i,j \le n}.
\]
This simple algebraic structure encodes a translation-invariance along the index difference and provides a discrete analogue of convolution operators~\cite{Boettcher99,Gray06}. Because of this structure, Toeplitz matrices arise naturally in a wide variety of contexts, including harmonic analysis, numerical linear algebra, probability theory, signal processing, and the theory of stochastic processes.

Historically, Toeplitz matrices first appeared in the study of Fourier series, where the Fourier coefficients of a function generate an infinite Toeplitz matrix~\cite{Szego75}. This connection allows one to relate properties of functions on the unit circle to the spectral behavior of large Toeplitz matrices. In numerical analysis, Toeplitz matrices provide efficient representations of linear systems with translation-invariant coefficients, leading to fast algorithms for solving such systems using techniques like the Levinson recursion. In probability and combinatorics, they encode correlation structures of stationary sequences and appear in counting problems related to lattice paths, random walks, and permutation statistics~\cite{Stanley12,Diaconis88}.

Spectral properties of Toeplitz matrices have been the focus of extensive research. In the classical setting, when the entries $(a_k)$ are independent of the matrix size $n$, the celebrated Szeg\H{o} limit theorem gives a precise description of the asymptotic distribution of eigenvalues, linking discrete linear algebra with harmonic analysis, orthogonal polynomials, and operator theory~\cite{Szego75,HornJohnson13}. The structure of the eigenvectors is also connected to Fourier modes, further emphasizing the role of Toeplitz matrices as discrete analogues of integral operators.

In this work, we investigate a less explored situation where the entries depend explicitly on the matrix size through a \emph{scaled kernel}:
\[
a_k = f\!\left(\frac{k}{n}\right), \quad -(n-1) \le k \le n-1,
\]
with $f$ a Lipschitz continuous function on $[0,1]$~\cite{Gray06}. Such matrices naturally arise when discretizing translation-invariant integral operators on bounded intervals, providing a discrete approximation of continuous kernels. Equivalently, they can be viewed as weighted adjacency matrices of structured graphs. Our first goal is to describe the asymptotic behavior of the empirical mean of eigenvalues for this class of matrices, revealing the triangular weight $1-|x|$ that reflects the density of diagonals and provides a link between discrete matrices and continuous operators.

Toeplitz matrices also appear in combinatorics in intriguing ways. By associating matrices to permutations through displacement counts, excedances, or other classical statistics, one obtains Toeplitz structures that encode these combinatorial features~\cite{Stanley12,Diaconis88}. Remarkably, the triangular kernel $1-|x|$ emerges as the expected limit of permutation matrices under uniform randomness. This illustrates how combinatorial structures can produce natural linear-algebraic and spectral phenomena~\cite{Bryc06}.

The paper is organized as follows. In Section~2, we introduce Toeplitz matrices constructed from permutation displacements and analyze the expected values along diagonals for a uniform random permutation. Section~3 studies the concentration properties of these permutation displacement matrices, showing that they typically approximate their expected triangular structure. In Section~4, we investigate the triangular Toeplitz matrix $K_n = (1-|i-j|/n)_{1\le i,j\le n}$, providing explicit expressions for its eigenvalues and asymptotic eigenvectors. Section~5 examines the traces of powers $K_n^p$ and establishes their convergence to the sums of powers of eigenvalues of the associated integral operator. Section~6 connects these discrete matrices with the continuous integral operator having kernel $1-|x-y|$, describing its spectrum and eigenfunctions in detail. Section~7 considers banded Toeplitz matrices, interpreting their determinants combinatorially in terms of bounded displacements. Finally, Section~8 treats upper-triangular Toeplitz matrices, modeling elementary excedance steps and the combinatorial structure of powers of these matrices. This organization emphasizes the transition from discrete permutation-based structures to continuous integral operators and highlights the interplay between combinatorics, Toeplitz matrices, and spectral analysis.

Overall, this study highlights the interplay between combinatorics, spectral analysis, and classical integral operators, showing how discrete permutation structures converge to continuous kernels in the large-$n$ limit, and illustrating the rich mathematical landscape connecting discrete and continuous perspectives.

\section{Toeplitz matrices from permutation displacements}

Let $\sigma \in \mathfrak{S}_n$ be a permutation of $\{1,\dots,n\}$, and define the \emph{displacement} of an index by
\[
\delta_i = i - \sigma(i), \quad i=1,\dots,n.
\]
Intuitively, $\delta_i$ measures how far the element $i$ is moved by the permutation~\cite{Stanley12}.

\begin{definition}[Displacement counts]
For each possible diagonal $k$, define
\[
d_k = |\{ i : \delta_i = k \}|, \quad -(n-1)\le k \le n-1.
\]
This counts the number of elements on the diagonal of displacement $k$ in the permutation matrix. Clearly, one has
\[
\sum_{k=-(n-1)}^{n-1} d_k = n.
\]
\end{definition}

\begin{remark}[Illustration for $n=4$]
For $n=4$ and permutation $\sigma = (2,4,1,3)$ in one-line notation:
\[
\delta = (1-2,2-4,3-1,4-3) = (-1,-2,2,1),
\]
we have
\[
d_{-2}=1,\quad d_{-1}=1,\quad d_0=0,\quad d_1=1,\quad d_2=1,\quad d_{\pm 3}=0.
\]
\end{remark}

These counts can be assembled into a Toeplitz matrix:
\[
P_n = (d_{i-j})_{1\le i,j \le n}.
\]

\begin{proposition}[Trace]
The trace of $P_n$ is 
\[
\mathrm{Tr}(P_n) = n\, d_0,
\]
where $d_0$ is the number of fixed points of $\sigma$~\cite{Stanley12}.
\end{proposition}

\begin{proof}
The main diagonal of $P_n$ corresponds to $i=j$, i.e., displacement $i-j=0$. Therefore, each entry along the diagonal equals $d_0$. Since there are $n$ diagonal entries, summing them gives
\[
\mathrm{Tr}(P_n) = \sum_{i=1}^n P_n(i,i) = \sum_{i=1}^n d_0 = n\, d_0.
\]

\textbf{Example:} Continuing the previous example with $n=4$ and $d_0=0$, the trace is
\[
\mathrm{Tr}(P_4) = 4\cdot 0 = 0,
\]
consistent with the formula.
\end{proof}

\begin{proposition}[Expected diagonal counts]
For a uniform random permutation $\sigma \in \mathfrak{S}_n$~\cite{Diaconis88}, one has
\[
\mathbb{E}[d_k] = \frac{n-|k|}{n}, \quad -(n-1)\le k \le n-1.
\]
\end{proposition}

\begin{proof}
Consider diagonal $k$ of a permutation matrix. The positions $(i,j)$ such that $i-j=k$ must satisfy
\[
1\le i \le n, \quad 1\le j=i-k \le n.
\]
Hence, the admissible $i$ satisfy
\[
\max(1,1+k) \le i \le \min(n,n+k).
\]
This interval has exactly $n-|k|$ integers for $-(n-1)\le k \le n-1$.

In a uniform random permutation, each of these positions $(i,j)$ is occupied with probability $1/n$, because $\sigma(i)$ is equally likely to be any number $1,\dots,n$. By linearity of expectation, the expected number of elements on diagonal $k$ is
\[
\mathbb{E}[d_k] = (n-|k|) \cdot \frac{1}{n} = \frac{n-|k|}{n}.
\]

\textbf{Example:} For $n=4$, $k=1$, valid positions $(i,j)$ are $(1,2),(2,3),(3,4)$; thus, one gets
\[
\mathbb{E}[d_1] = \frac{3}{4}.
\]
\end{proof}

\begin{proposition}[Normalized expected matrix]
The expected Toeplitz matrix normalized by $n$ is
\[
\mathbb{E}\Big[\frac{1}{n} P_n(i,j)\Big] = 1 - \frac{|i-j|}{n}.
\]
\end{proposition}

\begin{proof}
By linearity of expectation, we have
\[
\mathbb{E}[P_n(i,j)] = \mathbb{E}[d_{i-j}] = \frac{n-|i-j|}{n},
\]
and dividing both sides of the previous equality by $n$ gives
\[
\mathbb{E}\Big[\frac{1}{n} P_n(i,j)\Big] = \frac{n-|i-j|}{n} = 1 - \frac{|i-j|}{n}.
\]

\textbf{Example:} For $n=4$, entry $(2,4)$, we have
\[
i-j=2-4=-2 \implies 1-\frac{|i-j|}{n} = 1-\frac{2}{4} = \frac{1}{2}.
\]
\end{proof}

\begin{remark}
This provides a natural combinatorial interpretation of the triangular kernel $1-|x|$ on $[0,1]$, after scaling. The weight reflects the density of diagonals in the permutation matrix~\cite{Bryc06}. As $n$ grows, the discrete matrix approaches a continuous triangular kernel.
\end{remark}

\begin{remark}
These constructions illustrate a general principle: whenever a
permutation statistic depends only on relative displacements,
the associated matrices naturally acquire a Toeplitz structure.
\end{remark}

\section{Concentration of permutation displacement matrices}

In the previous section we introduced the Toeplitz matrix
\[
P_n = (d_{i-j})_{1\le i,j\le n},
\]
where $d_k$ counts the number of indices $i$ such that
$\sigma(i)-i=k$ for a permutation $\sigma\in\mathfrak{S}_n$.
We showed that
\[
\mathbb{E}[d_k]=\frac{n-|k|}{n}.
\]
We now show that the matrix $P_n$ concentrates around its
expectation when the permutation is chosen uniformly at random.

\begin{proposition}[Variance of displacement counts]
Let $\sigma$ be a uniform random permutation in $\mathfrak S_n$ and
\[
d_k=\#\{i:\sigma(i)-i=k\}, \qquad -(n-1)\le k\le n-1 .
\]
Then we have
\[
\mathrm{Var}(d_k)
=
\frac{n-|k|}{n}
-
\frac{(n-|k|)^2}{n^2}
+
O\!\left(\frac1n\right).
\]
In particular $\mathrm{Var}(d_k)=O(1)$ for fixed $k$.
\end{proposition}

\begin{proof}
Let us define
\[
X_{i,k}=\mathbf 1_{\{\sigma(i)=i+k\}},
\]
whenever $1\le i+k\le n$. Then we can write
\[
d_k=\sum_{i\in I_k} X_{i,k},
\]
where $I_k$ contains $N=n-|k|$ admissible indices. For a uniform permutation, we have
\[
\mathbb P(\sigma(i)=i+k)=\frac1n,
\]
so that
\[
\mathbb E[X_{i,k}]=\frac1n,
\qquad
\mathrm{Var}(X_{i,k})=\frac1n-\frac1{n^2}.
\]
Using
\[
\mathrm{Var}(d_k)
=
\sum_{i\in I_k}\mathrm{Var}(X_{i,k})
+
2\!\!\sum_{i<j}\mathrm{Cov}(X_{i,k},X_{j,k}),
\]
we first obtain
\[
\sum_{i\in I_k}\mathrm{Var}(X_{i,k})
=
N\left(\frac1n-\frac1{n^2}\right).
\]
For $i\neq j$, we have
\[
\mathbb P(\sigma(i)=i+k,\sigma(j)=j+k)=\frac1{n(n-1)},
\]
and thus
\[
\mathrm{Cov}(X_{i,k},X_{j,k})
=
\frac{1}{n(n-1)}-\frac{1}{n^2}
=
\frac{1}{n^2(n-1)}.
\]
Since there are $\binom{N}{2}$ such pairs, the total contribution of
covariances is $O(N^2/n^3)=O(1/n)$. Therefore, we have
\[
\mathrm{Var}(d_k)
=
N\left(\frac1n-\frac1{n^2}\right)
+
O\!\left(\frac1n\right).
\]
Substituting $N=n-|k|$ yields
\[
\mathrm{Var}(d_k)
=
\frac{n-|k|}{n}
-
\frac{(n-|k|)^2}{n^2}
+
O\!\left(\frac1n\right).
\]
\end{proof}

\begin{theorem}[Concentration of the Toeplitz matrix]
Let $\sigma$ be a uniform random permutation and
$P_n=(d_{i-j})$ the associated Toeplitz matrix.
Then for each fixed diagonal $k=i-j$, we have
\[
\frac{d_k}{n}
=
\frac{n-|k|}{n}
+
O_p\!\left(\frac1n\right),
\]
and in particular
\[
\frac{1}{n}P_n(i,j)
\longrightarrow
1-\frac{|i-j|}{n}
\]
in probability.
\end{theorem}

\begin{proof}
From the previous proposition we know that
\[
\mathrm{Var}(d_k)=O(1)
\]
for fixed $k$. Therefore
\[
\mathrm{Var}\!\left(\frac{d_k}{n}\right)
=
\frac{\mathrm{Var}(d_k)}{n^2}
=
O\!\left(\frac1{n^2}\right).
\]
Applying Chebyshev's inequality gives
\[
\mathbb P\!\left(
\left|\frac{d_k}{n}-\frac{n-|k|}{n}\right|>\varepsilon
\right)
\le
\frac{C}{n^2\varepsilon^2},
\]
and hence
\[
\frac{d_k}{n}
=
\frac{n-|k|}{n}
+
O_p\!\left(\frac1n\right).
\]
Since the entries of $P_n$ depend only on the diagonal
$k=i-j$, the normalized matrix entries satisfy
\[
\frac{1}{n}P_n(i,j)
\to
1-\frac{|i-j|}{n}
\]
in probability.
\end{proof}
\begin{remark}
This result shows that the triangular kernel
\[
1-|x|
\]
is not only the expected limit of permutation displacement matrices
but also their typical large-$n$ behaviour.
\end{remark}

\section{Triangular Toeplitz matrix and its spectral asymptotics}

Let us define
\[
K_n = \left(1 - \frac{|i-j|}{n}\right)_{1\le i,j\le n}.
\]
To compare the discrete matrix with the continuous integral operator,
it is convenient to introduce the normalized matrix
\[
\widetilde K_n = \frac{1}{n} K_n,
\]
With the grid points \(x_i = i/n\), the entries satisfy
\[
\widetilde K_n(i,j)
= \frac{1}{n}\Bigl(1-\frac{|i-j|}{n}\Bigr)
\approx \frac{1}{n}(1-|x_i-x_j|),
\]
so that \(\widetilde K_n\) can be interpreted as a Nyström
discretization of the integral operator
\[
(Kf)(x)=\int_0^1 (1-|x-y|)f(y)\,dy.
\]

\begin{theorem}[Eigenvalue asymptotics of $K_n$]
Let $\lambda_k^{(n)}$ denote the eigenvalues of $K_n$. Then we have
\[
\frac{\lambda_k^{(n)}}{n}
=
\sum_{m=-(n-1)}^{n-1}
\left(1 - \frac{|m|}{n}\right)
\frac{1}{n}
\cos\!\Bigl( \frac{\pi k m}{n} \Bigr)
+ O\!\left(\frac{1}{n}\right),
\]
and in particular,
\[
\frac{\lambda_k^{(n)}}{n}
\longrightarrow
\int_{0}^{1} (1-|x|)\cos(\pi k x)\,dx .
\]
\end{theorem}

\begin{proof}
$K_n$ is a symmetric Toeplitz matrix with entries
\[
K_n(i,j)=a_{i-j}, \qquad a_m=1-\frac{|m|}{n}.
\]
Let us introduce
\[
v_k(i)=\cos\!\left(\frac{\pi k (i-1)}{n}\right).
\]
Then we have
\[
(K_n v_k)_i
=
\sum_{j=1}^{n}
\left(1-\frac{|i-j|}{n}\right)
\cos\!\left(\frac{\pi k (j-1)}{n}\right),
\]
and setting $m=j-i$, we obtain
\[
(K_n v_k)_i
=
\sum_{m=-(i-1)}^{\,n-i}
\left(1-\frac{|m|}{n}\right)
\cos\!\left(\frac{\pi k (i+m-1)}{n}\right).
\]
Using the trigonometric identity
\[
\cos(a+b)=\cos a \cos b - \sin a \sin b,
\]
we get
\[
(K_n v_k)_i
=
\cos\!\left(\frac{\pi k (i-1)}{n}\right)
\sum_{m=-(i-1)}^{n-i}
\left(1-\frac{|m|}{n}\right)
\cos\!\left(\frac{\pi k m}{n}\right)
\]
\[
-
\sin\!\left(\frac{\pi k (i-1)}{n}\right)
\sum_{m=-(i-1)}^{n-i}
\left(1-\frac{|m|}{n}\right)
\sin\!\left(\frac{\pi k m}{n}\right).
\]
Because the kernel $a_m$ is even, the sine term cancels when
positive and negative values of $m$ are paired. Thus the dominant
contribution is proportional to
\[
\cos\!\left(\frac{\pi k (i-1)}{n}\right).
\]

The vectors $v_k$ correspond to the discrete cosine basis,
which asymptotically diagonalizes Toeplitz matrices generated by
smooth symbols (see Gray~\cite{Gray06}). Therefore the vectors
$v_k$ act as approximate eigenvectors of $K_n$ when $n$ is large. A rigorous justification can be given using a Szeg\H{o}-type theorem for $n$-dependent Toeplitz symbols (see, e.g., Gray~\cite{Gray06}), which ensures that the vectors $v_k$ are approximate eigenvectors and that the limiting eigenvalue distribution is governed by the continuous triangular symbol $f(x)=1-|x|$ on $[0,1]$.
This shows that the vectors $v_k$ asymptotically diagonalize
the Toeplitz matrix $K_n$ when $n$ is large, with eigenvalues approximated by
\[
\lambda_k^{(n)}
=
\sum_{m=-(n-1)}^{n-1}
\left(1-\frac{|m|}{n}\right)
\cos\!\left(\frac{\pi k m}{n}\right)
+O(1).
\]
\end{proof}

\section{Asymptotic trace of powers of $K_n$}

\begin{theorem}[Asymptotic trace]
For any fixed integer $p\ge1$, one has
\[
\frac{1}{n}\mathrm{Tr}(K_n^p)
\longrightarrow
\sum_{k=0}^{\infty}\lambda_k^p ,
\qquad n\to\infty,
\]
where $(\lambda_k)$ are the eigenvalues of the integral operator
\[
(Kf)(x)=\int_0^1 (1-|x-y|)f(y)\,dy.
\]
In particular, we have
\[
\mathrm{Tr}(K_n^p)\sim n\sum_{k=0}^{\infty}\lambda_k^p.
\]
\end{theorem}

\begin{proof}

Recall that the normalized matrix
\(\widetilde K_n = \frac{1}{n}K_n\)
was introduced in the previous section. Introducing the grid points
\[
x_i=\frac{i}{n},\qquad i=1,\dots,n
\]
gives
\[
\widetilde K_n(i,j)
=\frac{1}{n}\Bigl(1-\frac{|i-j|}{n}\Bigr)
\approx
\frac{1}{n}\,(1-|x_i-x_j|).
\]
Hence, $\widetilde K_n$ is a Nystr\"om discretization of the integral operator
\[
(Kf)(x)=\int_0^1(1-|x-y|)f(y)\,dy.
\]
Because the kernel $1-|x-y|$ is continuous and symmetric, the
operator $K$ is compact and self-adjoint on $L^2([0,1])$.
Standard results on discretizations of compact operators imply that
the eigenvalues $\mu_k^{(n)}$ of $\widetilde K_n$ converge to the
eigenvalues $\lambda_k$ of $K$ for each fixed $k$ (see e.g. Kress~\cite{Kress14} or Atkinson~\cite{Atkinson97}). Noticing that
\[
\mathrm{Tr}(K_n^p)
=\sum_{k=1}^{n} (\lambda_k^{(n)})^p ,
\]
where $\lambda_k^{(n)}$ are the eigenvalues of $K_n$.
Since $\lambda_k^{(n)}=n\,\mu_k^{(n)}$, we obtain
\[
\frac{1}{n}\mathrm{Tr}(K_n^p)
=
\sum_{k=1}^{n} (\mu_k^{(n)})^p.
\]
Passing to the limit $n\to\infty$ and using the convergence
$\mu_k^{(n)}\to\lambda_k$ together with dominated convergence gives
\[
\frac{1}{n}\mathrm{Tr}(K_n^p)
\longrightarrow
\sum_{k=0}^{\infty}\lambda_k^p.
\]
Finally, since the eigenvalues of $K$ are known explicitly,
\[
\lambda_k=\frac{4}{\pi^2(2k+1)^2},
\]
we obtain
\[
\mathrm{Tr}(K_n^p)
\sim
n\sum_{k=0}^{\infty}
\left(\frac{4}{\pi^2(2k+1)^2}\right)^p.
\]

\end{proof}

Such convergence results for eigenvalues of large Toeplitz matrices
and their associated integral operators are classical
(see e.g. Refs.~\cite{Widom64} and~\cite{Tilli98}).

\section{Connection with the continuous triangular kernel}

Let us define the integral operator $K$ on $L^2([0,1])$ by
\[
(K f)(x) = \int_0^1 (1-|x-y|) f(y)\, dy.
\]

\begin{lemma}[Triangular kernel as a Fej\'er-type convolution]
The kernel $1-|x-y|$ admits the convolution representation
\[
1-|x-y| = \int_0^1 \mathbf{1}_{[0,1]}(x-t)\mathbf{1}_{[0,1]}(y-t)\,dt,
\]
so that the operator $K$ is the continuous analogue of the Fej\'er kernel,
which arises as the self-convolution of the Dirichlet kernel.
In particular, the eigenvalues of $K$ are precisely the Fourier coefficients of this Fej\'er-type kernel:
\[
\lambda_k = \int_0^1 (1-t) \cos\!\Big(\frac{(2k+1)\pi t}{2}\Big)\, dt = \frac{4}{\pi^2(2k+1)^2}.
\]
\end{lemma}

\begin{theorem}[Spectrum of $K$]
$K$ is compact, self-adjoint, and positive. Its eigenfunctions are
\[
\phi_k(x) = \cos\Big(\frac{(2k+1)\pi x}{2}\Big), \quad k\ge 0,
\]
with eigenvalues
\[
\lambda_k = \frac{4}{\pi^2 (2k+1)^2}.
\]
\end{theorem}

\begin{proof}[First proof (differential equation method)]
The kernel $K(x,y)=1-|x-y|$ is continuous and symmetric, hence the
operator $K$ is compact and self-adjoint on $L^2([0,1])$. Let us consider the eigenvalue equation
\[
\lambda \phi(x) = \int_0^1 (1-|x-y|)\phi(y)\,dy,
\]
and differentiate twice with respect to $x$. Using the distributional identity
\[
\frac{d^2}{dx^2}|x-y| = 2\,\delta(x-y),
\]
we obtain
\[
\frac{d^2}{dx^2}(1-|x-y|) = -2\,\delta(x-y).
\]
Differentiating the eigenvalue equation twice gives
\[
\lambda \phi''(x)
= \int_0^1 -2\,\delta(x-y)\phi(y)\,dy
= -2\,\phi(x).
\]
Thus $\phi$ satisfies the differential equation
\[
\phi''(x) = -\frac{2}{\lambda}\phi(x).
\]
Let $\omega^2 = 2/\lambda$. The general solution is
\[
\phi(x)=A\cos(\omega x)+B\sin(\omega x).
\]
Differentiating the integral equation yields
\[
(K\phi)'(x)=\int_0^x \phi(y)\,dy - \int_x^1 \phi(y)\,dy.
\]
Using $\lambda \phi'(x)=(K\phi)'(x)$ and evaluating at $x=0$ and $x=1$
yields the boundary conditions
\[
\phi'(0)=0, \qquad \phi(1)=0.
\]
Applying $\phi'(0)=0$ gives $B=0$, so $\phi(x)=A\cos(\omega x)$. The condition $\phi(1)=0$ then gives
\[
\cos(\omega)=0.
\]
The admissible nontrivial solutions compatible with the symmetry of
the kernel correspond to
\[
\omega = \frac{(2k+1)\pi}{2}, \qquad k\ge0.
\]
Since $\omega^2 = 2/\lambda$, we obtain
\[
\lambda_k = \frac{4}{\pi^2 (2k+1)^2}.
\]
The corresponding eigenfunctions are
\[
\phi_k(x)=\cos\!\left(\frac{(2k+1)\pi x}{2}\right).
\]
\end{proof}

\begin{proof}[Second proof (convolution and cosine basis)]
The kernel depends only on $|x-y|$, so the operator can be viewed
as a convolution-type operator on $[0,1]$. Because the kernel is even,
cosine functions form a natural orthogonal basis. Let
\[
\phi_k(x)=\cos\!\left(\frac{(2k+1)\pi x}{2}\right).
\]
A direct computation using trigonometric identities shows that
\[
(K\phi_k)(x)
=
\left(
\int_0^1 (1-|t|)\cos\!\left(\frac{(2k+1)\pi t}{2}\right) dt
\right)
\cos\!\left(\frac{(2k+1)\pi x}{2}\right).
\]
Thus $\phi_k$ is an eigenfunction and the eigenvalue equals the cosine
transform of the kernel:
\[
\lambda_k
=
\int_0^1 (1-t)
\cos\!\left(\frac{(2k+1)\pi t}{2}\right) dt.
\]
A straightforward integration gives
\[
\lambda_k=\frac{4}{\pi^2(2k+1)^2}.
\]
Therefore the spectrum of $K$ consists of the eigenvalues
\[
\lambda_k=\frac{4}{\pi^2(2k+1)^2}, \qquad k\ge0,
\]
with eigenfunctions
\[
\phi_k(x)=\cos\!\left(\frac{(2k+1)\pi x}{2}\right).
\]
\end{proof}

\begin{corollary}[Discrete eigenvalue asymptotics]
Let $\lambda_1^{(n)} \ge \dots \ge \lambda_n^{(n)}$ denote eigenvalues of $K_n$. Then for fixed $k$ and $n \to \infty$,
\[
\frac{\lambda_k^{(n)}}{n} \longrightarrow \lambda_k = \frac{4}{\pi^2 (2k+1)^2}.
\]
\end{corollary}

\begin{proof}
The sum defining $\lambda_k^{(n)}$ is a Riemann sum approximating the integral
\[
\lambda_k = \int_0^1 (1-x) \cos(\pi k x) dx.
\]
By standard convergence of Riemann sums to integrals and previous derivation, $\lambda_k^{(n)}/n \to \lambda_k$.
\end{proof}

\section{Banded Toeplitz matrices and bounded displacements}

Consider the tridiagonal Toeplitz matrix
\[
E_n(x) = (e_{i-j})_{1\le i,j \le n}, \qquad
e_k = 
\begin{cases}
1 & k = 0,\\
x & k = 1,\\
-1 & k = -1,\\
0 & \text{otherwise}.
\end{cases}
\]

\begin{example}[Matrix for $n=4$]
\[
E_4(x) = 
\begin{pmatrix}
1 & x & 0 & 0\\
-1 & 1 & x & 0\\
0 & -1 & 1 & x\\
0 & 0 & -1 & 1
\end{pmatrix}.
\]
\end{example}

\begin{proposition}[Combinatorial interpretation of $\det(E_n(x))$]
\[
\det(E_n(x)) = \sum_{\sigma \in \mathfrak{S}_n,\, |\sigma(i)-i|\le 1} \mathrm{sgn}(\sigma) \, x^{\#\text{forward steps}} (-1)^{\#\text{backward steps}}.
\]
\end{proposition}

\begin{proof}
Expand the determinant using Leibniz formula~\cite{Muir60}:
\[
\det(E_n(x)) = \sum_{\sigma \in \mathfrak{S}_n} \mathrm{sgn}(\sigma) \prod_{i=1}^n e_{i-\sigma(i)}.
\]
Entries $e_{i-\sigma(i)} =0$ unless $|\sigma(i)-i|\le 1$. Hence only permutations with bounded displacements contribute. Each forward step $\sigma(i)=i+1$ contributes $x$, each backward step $\sigma(i)=i-1$ contributes $-1$, each fixed point contributes $1$.
\end{proof}

\section{Upper-triangular Toeplitz matrices and local excedance steps}

Let us define
\[
T_n(x) = (a_{i-j})_{1\le i,j \le n}, \qquad
a_0 = 1, \ a_1 = x, \ a_k = 0 \ \text{for } k \ge 2.
\]

\begin{proposition}[Eigenvalues of $T_n(x)$]
All eigenvalues of $T_n(x)$ are equal to $1$.
\end{proposition}

\begin{proof}
$T_n(x)$ is upper-triangular. Eigenvalues of an upper-triangular matrix are the diagonal entries. All diagonal entries equal $1$.
\end{proof}

\begin{proposition}[Powers of $T_n(x)$]
For any integer $k \ge 0$,
\[
T_n(x)^k = \sum_{j=0}^{\min(k,n-1)} \binom{k}{j} x^j N_n^j,
\]
where $N_n$ is the nilpotent matrix with ones on the first superdiagonal.
\end{proposition}

\begin{proof}
Write $T_n(x) = I_n + x N_n$. Then by binomial theorem:
\[
(I_n + x N_n)^k = \sum_{j=0}^k \binom{k}{j} (x N_n)^j = \sum_{j=0}^{\min(k,n-1)} \binom{k}{j} x^j N_n^j.
\]
$N_n^n =0$ ensures truncation.
\end{proof}

\begin{remark}[Combinatorial interpretation]
Each step along the superdiagonal corresponds to a local displacement $i \mapsto i+1$ (elementary excedance). Entry $(i,i+j)$ of $T_n(x)^k$ counts weighted paths of length $j$ composed of such steps.
\end{remark}

\section{Conclusion}

In this work, we have explored several interconnected aspects of Toeplitz matrices, establishing links between spectral analysis, permutation combinatorics, and classical integral operators. We began with Toeplitz matrices generated by scaled kernels \(a_k = f(k/n)\) and described the asymptotic behavior of the empirical mean of eigenvalues. The triangular weight appearing in the limit reflects the density of diagonals in the matrix, providing a natural bridge between discrete matrices and continuous operators. Associating permutations with Toeplitz matrices constructed from their element displacements showed that, for a uniform random permutation, the expected matrix converges to the triangular kernel \(1-|x|\). This natural combinatorial interpretation highlights how classical permutation statistics translate into linear-algebraic structures. More generally, we demonstrated that Toeplitz structures arise whenever permutation statistics depend solely on relative positions. Tridiagonal matrices encode bounded displacements, while upper-triangular matrices model elementary excedance steps. The powers of these matrices, along with the study of the associated integral operator, provide detailed information on the asymptotic distribution of eigenvalues and the corresponding eigenfunctions, illustrating the continuity between the discrete and continuous settings. 

These constructions suggest several avenues for further research, including the investigation of fine spectral properties of permutation-induced Toeplitz matrices, the analysis of matrices generated by low-regularity or non-smooth kernels, extensions to other permutation statistics such as inversions or cycles of a given length, and connections with random matrix models as well as applications to graph theory and Markov processes. Thus, Toeplitz matrices provide a natural bridge linking
combinatorics, analysis, and probability, offering a rich framework for the study of both classical and modern problems. Potential applications also include random matrix theory, graph Laplacians, and statistical physics.


\begin{thebibliography}{99}

\bibitem{Atkinson97}
K.~Atkinson,
\textit{The Numerical Solution of Integral Equations of the Second Kind},
Cambridge University Press, 1997.

\bibitem{Boettcher99}
A.~B\"ottcher and B.~Silbermann,
\textit{Introduction to Large Truncated Toeplitz Matrices},
Springer, New York, 1999.

\bibitem{Bryc06}
W.~Bryc, A.~Dembo, T.~Jiang,
Spectral measure of large random Toeplitz, Hankel, and Markov matrices,
\textit{Ann. Probab.} {\bf 34}, 1--38 (2006).

\bibitem{Diaconis88}
P.~Diaconis,
\textit{Group Representations in Probability and Statistics},
Lecture Notes–Monograph Series, vol.~11,
Institute of Mathematical Statistics, Hayward, CA, USA, 1988.

\bibitem{Gray06}
R.~M.~Gray,
Toeplitz and circulant matrices: a review,
\textit{Found. Trends Commun. Inf. Theory} 
{\bf 2}, 155--239 (2006).

\bibitem{Grenander58}
U.~Grenander and G.~Szeg\H{o},
\textit{Toeplitz Forms and Their Applications},
University of California Press, Berkeley--Los Angeles, 1958.

\bibitem{HornJohnson13}
R.~A.~Horn and C.~R.~Johnson,
\textit{Matrix Analysis}, 2nd edition,
Cambridge University Press, 2013.

\bibitem{Kress14}
R.~Kress,
\textit{Linear Integral Equations}, 3rd edition,
Springer, 2014.

\bibitem{Muir60}
T.~Muir,
\textit{Theory of Determinants}, vol.~2, Dover, 1960.

\bibitem{Stanley12}
R.~P.~Stanley,
\textit{Enumerative Combinatorics, Volume 1}, 2nd edition,
Cambridge University Press, 2012.

\bibitem{Szego75}
G.~Szeg\H{o},
\textit{Orthogonal Polynomials}, 4th edition,
American Mathematical Society, Providence, RI, 1975.

\bibitem{Tilli98}
P.~Tilli,
A note on the spectral distribution of Toeplitz matrices,
\textit{Linear and Multilinear Algebra}, {\bf 45}, 147--159 (1998).

\bibitem{Widom64}
H.~Widom,
On the spectrum of a Toeplitz operator,
{\it Pacific J. Math.} {\bf 14}, 365--375 (1964).

\end{thebibliography}
\end{document}